\documentclass[11pt]{amsart}
\usepackage{latexsym}
\usepackage{amssymb}
\usepackage{amsmath}

\usepackage[pdftex,colorlinks]{hyperref}

\usepackage{amsfonts}

\newtheorem{theorem}{Theorem}[section]
\newtheorem{lemma}[theorem]{Lemma}
\newtheorem{proposition}[theorem]{Proposition}

\newtheorem{example}[theorem]{Example}

\newtheorem{definition}[theorem]{Definition}
\newtheorem{corollary}[theorem]{Corollary}
\newtheorem{remark}[theorem]{Remark}

\newcommand\sspp{\mathop{\rm span}}
\newcommand\vn{\mathop{\rm VN}}

\newcommand\Bim{\mathop{\rm Bim}}
\newcommand\Sat{\mathop{\rm Sat}}

\newcommand\nul{\mathop{\rm null}}

\newcommand{\cl}[1]{\mathcal{#1}}
\newcommand{\bb}[1]{\mathbb{#1}}
\newcommand{\du}[2]{\left\langle{#1},{#2} \right\rangle} %\du{T}{u}
\newcommand{\nor}[1]{\left\Vert #1\right\Vert}    %\nor{x^2}
\newcommand{\sca}[1]{\left(#1\right)} %

\newcommand\an{^{-1}}

\newcommand\cb{\mathop{\rm cb}}

\def\gl{\lambda}

\def\gs{\sigma}

\def\Gs{\Sigma}
\def\go{\omega}
\def\Go{\Omega}

\newcommand\suppG{\mathop{\rm supp_{\it G}}}
\newcommand\suppo{\mathop{\rm supp_{\omega}}}

\begin{document}

\title[Ideals, supports and harmonic operators]
{Ideals of the Fourier algebra, supports and harmonic operators}

\date{}
%%%finharm5 compiled July 12, 2014%%%%%%%%%%

\author{M. Anoussis, A. Katavolos and  I. G. Todorov}

\address{Department of Mathematics, University of the Aegean,
Samos 83 200, Greece}

\email{mano@aegean.gr}

\address{Department of Mathematics, University of Athens,
Athens 157 84, Greece}

\email{akatavol@math.uoa.gr}

\address{Pure Mathematics Research Centre, Queen's University Belfast,
Belfast BT7 1NN, United Kingdom}

\email{i.todorov@qub.ac.uk}


\begin{abstract}
We examine the common null spaces of families of Herz-Schur multipliers and 
apply our results to study jointly harmonic operators and their relation with 
jointly harmonic functionals. We show how an annihilation formula obtained in 
\cite{akt} can be used to give a short proof {as well as a generalisation} 
of a result of Neufang and Runde  
concerning harmonic operators with respect to a normalised positive definite function. 
We compare the two notions of support of an operator that have been 
studied in the literature and show how one can be expressed in terms of the other. 
\end{abstract}

\maketitle

\section{Introduction and Preliminaries}

In this paper we investigate, for a locally compact group $G$, 
the common null spaces of families of Herz-Schur multipliers 
(or completely bounded multipliers of the Fourier algebra $A(G)$) 
and their relation to ideals of $A(G)$. 

This provides a new perspective for our previous results in \cite{akt}
concerning (weak* closed) spaces of operators on $L^2(G)$ which are simultaneously
invariant under all Schur multipliers and under 
{conjugation by the right regular  representation}
of $G$ on $L^2(G)$ ({\em jointly invariant} subspaces -- see below 
for precise definitions).

At the same time, it provides a new approach to, as well as an extension 
of, a result of  Neufang and Runde \cite{neurun} concerning the space
 $\widetilde{\cl H}_\gs$ 
of operators which are `harmonic' with respect  to 
a positive definite normalised function $\gs:G\to\bb C$. The notion of $\gs$-harmonic
operators was 
introduced in \cite{neurun} as an extension of the notion of $\gs$-harmonic 
functionals on $A(G)$ as defined and studied by Chu and Lau in \cite{chulau}. 
One of the main results of  Neufang and Runde is that $\widetilde{\cl H}_\gs$ 
is the von Neumann algebra
on $L^2(G)$ generated 
by the algebra $\cl D$ of multiplication operators together 
with the space ${\cl H}_\gs$ of harmonic functionals, 
considered as a subspace of the von Neumann algebra $\vn(G)$ of the group.

It will be seen that this result can be obtained as a consequence of the fact 
(see Corollary \ref{c_jho})
that, for any family $\Gs$ of completely bounded multipliers of $A(G)$,     
the space $\widetilde{\cl H}_\Gs$ of {\em jointly $\Gs$-harmonic operators} 
can be obtained as the weak* closed $\cl D$-bimodule 
generated by the {\em jointly $\Gs$-harmonic functionals} ${\cl H}_\Gs$.
In fact, the spaces  $\widetilde{\cl H}_\Gs$ belong to the class 
of jointly invariant subspaces of $\cl B(L^2(G))$ studied 
in \cite[Section 4]{akt}.

The space ${\cl H}_\Gs$ is the annihilator in $\vn(G)$ of a certain ideal of $A(G)$.  
Now from any given closed ideal $J$  of the Fourier algebra $A(G)$, 
there are two `canonical' ways to arrive at 
a weak* closed $\cl D$-bimodule of 
$\cl B(L^2(G))$. One way is to consider its annihilator $J^\perp$ 
in $\vn(G)$ and then take the weak* closed 
$\cl D$-bimodule
generated by $J^{\perp}$. We call this bimodule $\Bim(J^\perp)$. The other way is to
take a suitable saturation $\Sat(J)$ of $J$ 
within the trace class operators 
on $L^2(G)$ (see Theorem \ref{th_satlcg}), and then form its annihilator. This gives 
a masa bimodule $(\Sat J)^{\perp}$ in $\cl B(L^2(G))$. In \cite{akt},
we proved that these
 two procedures  yield the same bimodule, that is, $\Bim(J^\perp) = (\Sat J)^{\perp}$. 
Our proof  that $\widetilde{\cl H}_\Gs=\Bim({\cl H}_\Gs)$ 
rests on this equality. 

The notion of {\em support}, $\suppG(T)$,  of an element $T\in\vn(G)$ was 
introduced by Eymard in \cite{eymard} by considering $T$ as a linear functional
on the function algebra $A(G)$; thus  $\suppG(T)$ is a closed subset of $G$. 
This notion was extended by Neufang and Runde in \cite{neurun} to an arbitrary
 $T\in\cl B(L^2(G))$ 
and used to describe harmonic operators. By considering joint supports,
we show that this extended notion of $G$-support for an operator $T\in\cl B(L^2(G))$
coincides
with the joint $G$-support of a family of elements of $\vn (G)$ naturally associated    
to $T$ (Proposition \ref{propsame2}).

On the other hand, the notion of support of an operator $T$ acting on $L^2(G)$ 
was first introduced by  Arveson in \cite{arv}
as a certain closed subset of $G \times G$.
This notion was used 
in his study of what was later called operator synthesis. 
A different but related approach  appears in \cite{eks}, 
where the notion of 
$\omega$-support, $\suppo(T)$, of $T$  was introduced
and used to establish a bijective correspondence between 
reflexive masa-bimodules and $\go$-closed subsets of $G\times G$. 

We show that the joint $G$-support $\suppG(\cl A)$ of an arbitrary family  
 $\cl A\subseteq \cl B(L^2(G))$ can be fully described in terms of its
joint $\go$-support $\suppo(\cl A)$ (Theorem \ref{th_compsa}). 
The converse does not hold in general, 
as the $\go$-support, being a subset of $G\times G$, contains in general more
 information about an arbitrary operator than its $G$-support 
(see Remark \ref{last});
however, in case $\cl A$ is a (weak* closed) jointly invariant subspace,   
we show that its $\go$-support can be recovered from its $G$-support
(Theorem \ref{312}).
We also show that, if a set  $\Omega\subseteq G\times G$ is invariant 
under all maps $(s,t)\to (sr,tr), \, r\in G$, 
then $\Omega$ is marginally equivalent to an $\omega$-closed set if and only if 
it is marginally equivalent to a (topologically) closed set. 
This can fail for non-invariant sets (see for example  \cite[p. 561]{eks}). 
{For a related result, see  \cite[Proposition 7.3]{stt_clos}.}

\medskip

\noindent\textbf{Preliminaries and Notation }
Throughout, $G$ will denote a second countable locally compact group,
equipped with left Haar measure.
Denote by  $\cl D\subseteq\cl{B}(L^2(G))$ the maximal abelian selfadjoint algebra
(masa, for short) consisting of all  multiplication operators $M_f:g\to fg$, where 
$f\in L^\infty(G)$.
We write $\vn (G)$ for the von Neumann algebra $\{\gl_s : s\in G\}''$ generated by 
the left regular representation  $s\to \lambda_s$  of $G$ on  $L^2(G)$
(here $(\gl_sg)(t)=g(s\an t)$).

Every element of the predual of $\vn (G)$ is a vector functional, 
$\omega_{\xi,\eta}: T\to (T\xi,\eta)$,
where $\xi,\eta\in L^2(G)$, and $\nor{\go_{\xi,\eta}}$ is  the infimum of the products 
$\|\xi\|_2\|\eta\|_2$ over all such representations. This 
predual can be identified \cite{eymard} with the set $A(G)$ of all 
complex functions $u$ on $G$ of the form $s\to u(s)=\omega_{\xi,\eta}(\gl_s)$.
With the above norm and pointwise operations, $A(G)$  
is a (commutative, regular, semi-simple) Banach algebra of continuous functions 
on $G$ vanishing at infinity,
called the \emph{Fourier algebra} of $G$; its Gelfand spectrum can 
be identified with $G$ {\it via} point evaluations. 
The set $A_c(G)$ of compactly supported 
elements of $A(G)$  is dense in $A(G)$.

A function $\gs:G\to\bb C$ is a {\em multiplier} of $A(G)$ if for all $u\in A(G)$
the pointwise product $\gs u$ is again in $A(G)$. By duality, a multiplier $\gs$
induces a bounded operator $T\to \gs\cdot T$ on $\vn(G)$. We say $\gs$
is {\em a completely bounded (or Herz-Schur) multiplier}, 
and write $\gs\in M^{\cb}A(G)$,
if the latter operator is completely bounded, that is, if there exists a constant $K$ 
such that $\nor{[\gs\cdot T_{ij}]}\le K \nor{[T_{ij}]}$ for all $n\in\bb N$ and all 
$[T_{ij}]\in M_n(\vn (G))$ (the latter being the space of all $n$ by $n$ matrices with
entries in $\vn (G)$).
The least such constant is the \emph{cb norm} of $\gs$. 
The space  $M^{\cb}A(G)$ with pointwise operations and the cb norm is 
a Banach algebra into which $A(G)$ embeds contractively.  
For a subset $\Sigma\subseteq M^{\cb}A(G)$, we let 
$Z(\Sigma)=\{s\in G: \gs(s)=0 \text{ for all } \gs\in\Gs\}$
be its \emph{zero set}.

A subset $\Go\subseteq G\times G$ is called {\em marginally null} if 
there exists a null set (with respect to Haar measure) $X\subseteq G$  such that 
$\Go\subseteq (X\times G)\cup(G\times X)$. 
Two sets $\Go,\Go'\subseteq G\times G$ are {\em marginally equivalent}
if their symmetric difference is a marginally null set; 
we write $\Omega_1\cong \Omega_2$. 
A set $\Go\subseteq G\times G$ is said to be {\em $\go$-open} if 
it is marginally equivalent to a {\em countable} union of Borel rectangles $A\times B$;
it is called   {\em $\go$-closed} when its complement is $\go$-open.

Given any set  $\Go\subseteq G\times G$, we denote by  
$\frak M_{\max}(\Go)$ the set of all $T\in\cl{B}(L^2(G))$ which are {\em supported}
by $\Go$ in the sense that $M_{\chi_ B}TM_{\chi_A}=0$ whenever 
$A\times B\subseteq G\times G$
is a Borel rectangle disjoint from $\Go$
(we write $\chi_A$ for the characteristic function of a set $A$).
Given any set $\cl U\subseteq \cl{B}(L^2(G))$ there exists a smallest, up to marginal
equivalence, $\go$-closed set  $\Go\subseteq G\times G$ supporting every element
of $\cl U$, {\it i.e.} such that $\cl U\subseteq\frak M_{\max}(\Go)$. This set is 
called {\em the  $\go$-support} of $\cl U$ and is denoted $\suppo(\cl U)$ \cite{eks}.

Two functions $h_1,h_2 : G\times G\to \bb{C}$ are said to be 
{\em marginally equivalent}, or equal
{\em marginally almost everywhere (m.a.e.)},  if 
they differ on a marginally  null set.

The predual of  $\cl{B}(L^2(G))$ consists of all linear forms $\go$ given by 	
$\go(T):= \sum\limits_{i=1}^{\infty} \sca{Tf_i, g_i}$
where $f_i, g_i\in L^2(G)$ and $\sum\limits_{i=1}^{\infty}\nor{f_i}_2\nor{g_i}_2<\infty$.
Each such $\go$ defines a trace class operator whose kernel is a function
$h = h_\go:G\times G\to\bb C$,  unique up to marginal equivalence,  given by
$h(x,y)=\sum\limits_{i=1}^{\infty} f_i(x )\bar g_i(y)$. 
This series converges marginally almost everywhere on $G\times G$.
We use the notation $\du{T}{h} :=\go(T)$.

We write $T(G)$ for the Banach space of (marginal equivalence classes of) such 
functions, equipped  with the norm of the predual of  $\cl{B}(L^2(G))$.  

Let $\frak{S}(G)$ be the multiplier algebra of $T(G)$; by definition, a
measurable function $w : G\times G\rightarrow \bb{C}$ belongs to $\frak{S}(G)$ if
the map $m_w: h\to wh$ leaves $T(G)$ invariant, that is, if
$wh$ is marginally equivalent to a function from $T(G)$, for every $h\in T(G)$. 
Note that the operator $m_w$ is automatically bounded. 
The elements of $\frak{S}(G)$ are called \emph{(measurable) Schur multipliers}.
By duality, every Schur multiplier induces a bounded operator 
$S_w$ on $\cl B(L^2(G))$, given by 
\[\du{S_w(T)}{h} = \du{T}{wh}, \ \ \ h\in T(G), \; T\in \cl B(L^2(G))\, .\]
The operators of the form $S_w$, $w\in \frak{S}(G)$, are precisely the 
bounded weak* continuous $\cl D$-bimodule maps on $\cl B(L^2(G))$ 
(see \cite{haa}, \cite{sm}, \cite{pe} and \cite{kp}).

A weak* closed subspace $\cl U$ of $\cl B(L^2(G))$ is invariant under 
the maps $S_w$, $w\in \frak{S}(G)$, if and only if it is invariant under 
all left and right multiplications by elements of $\cl D$,
{\it i.e.} if $M_fTM_g\in \cl U$ for all $f,g\in L^\infty(G)$ and all $T\in\cl U$, in 
other words, if it is a {\em $\cl D$-bimodule}.   
For any set $\cl T\subseteq \cl B(L^2(G))$
we denote by $\Bim\cl T$ the smallest weak* closed $\cl D$-bimodule containing 
$\cl T$; 
thus, $\mathrm{Bim}(\cl T)=\overline{[\mathfrak{S}(G)\cl T]}^{w^*}$.

We call a subspace $\cl U\subseteq \cl B(L^2(G))$ {\em invariant} 
if $\rho_rT\rho_r^*\in\cl A$ for all $T\in\cl A$ and all $r\in G$; here, $r\to \rho_r$ is 
the right  regular representation of $G$ on $L^2(G)$. 
An invariant space, which is also a $\cl D$-bimodule, will be called a
{\em jointly invariant space}.
  
It is not hard to see that, if $\cl A\subseteq \cl B(L^2(G))$, 
the smallest weak* closed jointly invariant space 
 containing $\cl A$ is   the weak* closed linear span of 
$\{S_w(\rho_rT\rho_r^*): T\in\cl A, w\in \frak S(G), r\in G\}$.

For a complex function $u$ on $G$ we let $N(u):G\times G\to\bb C$ be
the function given by $N(u)(s,t) = u(ts^{-1})$. For any subset $E$ of $G$, we write 
$E^*=\{(s,t)\in G\times G: ts^{-1}\in E\}$. 

It is shown in \cite{bf} (see also \cite{j} and \cite{spronk}) 
that the map $u\rightarrow N(u)$ is an isometry 
from $M^{\cb}A(G)$ into $ \frak{S}(G)$ and that its range consists precisely of all
{\em invariant} Schur multipliers, {\it i.e.} those 
$w\in \frak{S}(G)$ for which $w(sr,tr) = w(s,t)$ for every $r\in G$ and
marginally almost all $s,t$. Note that the corresponding operators $S_{N(u)}$ 
are denoted $\hat\Theta(u)$ in \cite{neuruaspro}.

The following result from \cite{akt} is crucial for what follows.

\begin{theorem}\label{th_satlcg}
Let $J\subseteq A(G)$ be a closed ideal and
$\Sat(J)$ be the closed $L^\infty(G)$-bimodule of $T(G)$ generated by  the set
\[
\{N(u)\chi_{L\times L}: u \in J, L\  \text{compact, } \ L\subseteq G \}.
\]
Then
$\Sat(J)^{\perp} = \Bim (J^{\perp})$.
\end{theorem}

\section{Null spaces and harmonic operators}\label{s1}

Given a subset  $\Gs\subseteq M^{\cb}A(G)$, let
 \[ 
\frak{N}(\Sigma) = \{T\in \vn(G) : \gs\cdot T = 0, \ \mbox{ for all } \gs\in \Gs\}
\]
be the {\em common null set} of the operators on $\vn(G)$ of the form 
$T\to \sigma\cdot T$, with $\sigma\in \Sigma$. 
Letting
\[\Gs A \stackrel{def}{=}  
\overline{\sspp}(\Gs A(G)) = \overline{\sspp\{ \gs u : \gs\in \Sigma, u\in A(G)\}},\]
it is easy to verify that $\Gs A$ is a closed ideal of $A(G)$ and that 
\begin{equation}\label{eq_prean}
\frak{N}(\Sigma) = (\Gs A)^\bot .
\end{equation}

{\remark \label{remideal}
The sets of the form $\Gs A$ are precisely the closed ideals of $A(G)$
generated by their compactly supported elements.} 
\begin{proof}
It is clear that, if $\Gs\subseteq M^{\cb}A(G)$, the set 
$\{\gs u: \gs\in\Gs, u\in A_c(G)\}$ consists of compactly supported elements 
and is dense in  $\Gs A$. Conversely, suppose that 
$J\subseteq A(G)$ is a closed ideal such that $J\cap A_c(G)$ is dense in $J$. 
For every  $u\in J$ with compact support $K$, 
there exists $v\in A(G)$ which equals 1 on $K$ \cite[(3.2) Lemme]{eymard}, 
and so $u=uv\in JA$.
Thus $J= \overline{J\cap A_c(G)}\subseteq JA\subseteq J$ and hence $J=JA$.
\end{proof}

The following Proposition shows that 
it is sufficient to study sets of the form $\frak N(J)$ where $J$ is a closed ideal 
of $A(G)$.

\begin{proposition}\label{p_njan}  For any subset $\Gs$ of $M^{\cb}A(G)$, 
\[  \frak{N}(\Gs)=\frak{N}(\Gs A).\]
\end{proposition}
\proof 
If $\gs\cdot T = 0$ for all $\gs\in\Gs$ then {\em a fortiori} 
$v\gs\cdot T=0$, for all $v\in A(G)$ and all $\sigma\in \Sigma$.
It follows that $w\cdot T=0$ for all $w\in \Gs A$; thus 
$\frak{N}(\Gs)\subseteq \frak{N}(\Gs A)$.

Suppose conversely that $w\cdot T=0$ for all $w\in \Gs A$ and fix $\gs\in\Gs$. 
Now $u\gs\cdot T=0$ for all $u \in A(G)$,
and so $ \du{\gs\cdot T}{uv}=0$ when $u,v\in A(G)$. 
Since the products $uv$ form a dense subset of $A(G)$, we have   
$\gs\cdot T=0$. Thus $\frak{N}(\Gs)\supseteq \frak{N}(\Gs A)$ 
since $\gs\in\Gs$ is arbitrary, and the proof is complete. \qed 

\bigskip

It is not hard to see that $\gl_s$
is in $\frak N(\Gs)$ if and only if $s$ is in the zero set $Z(\Gs)$ of $\Gs$, and so 
$Z(\Gs)$ coincides with the zero set of the ideal $J=\Gs A$.
Whether or not, for an ideal $J$, these unitaries suffice to generate $\frak N(J)$
 depends on properties of the zero set.
  
For our purposes, a closed subset $E\subseteq G$ is a {\em set of synthesis} 
if there is a unique closed ideal $J$ of $A(G)$ with $Z(J)=E$. Note 
that this ideal is generated by its compactly supported elements 
\cite[Theorem 5.1.6]{kaniuth}.

\begin{lemma} \label{proto}
Let $J\subseteq A(G)$ be a closed ideal. Suppose that its zero set $E=Z(J)$ is 
a set of synthesis. Then
\[
\frak N(J)=J^\bot=\overline{\sspp\{\gl_x:x\in E\}}^{w*}
\]
\end{lemma}

\proof 
Since $E$ is a set of synthesis,  $J=JA$ by Remark \ref{remideal}; 
thus $J^\bot=(JA)^\bot=\frak N(J)$  by relation (\ref{eq_prean}). 
The other equality is essentially a reformulation of the fact that $E$ 
is a set of synthesis:  a function $u\in A(G)$ is in $J$ if and only if 
it vanishes at every point of $E$, that is, if and only if it annihilates every $\gl_s$
with $s\in E$ (since $\du{\gl_s}{u}= u(s)$).    \qed

\medskip 

A linear space $\cl U$ of bounded operators on a Hilbert space is called  
{\em a ternary ring of operators (TRO)} 
 if it satisfies $ST^*R\in\cl U$ whenever $S,T$ and $R$ are in $\cl U$. Note that a TRO
containing the identity operator is automatically a selfadjoint algebra.

\begin{proposition}\label{deutero}
Let $J\subseteq A(G)$ be a closed ideal. Suppose that its zero set $E=Z(J)$ is 
the  coset of a closed subgroup of $G$. Then
$\frak N(J)$ is a (weak-* closed) TRO. In particular, if $E$ is a closed subgroup
then  $\frak N(J)$ is a von Neumann subalgebra of $\vn (G)$.
\end{proposition}

\proof We may write $E=Hg$ where $H$ is a closed subgroup and $g\in G$
(the proof for the case  $E=gH$ is identical).
Now $E$ is a translate of $H$ which is a set of synthesis by \cite{tatsuuma2}
and hence $E$ is a set of synthesis. 
Thus Lemma \ref{proto} applies.

If $sg,tg,rg$ are in $E$ and $S=\gl_{sg},  \, T=\gl_{tg}$  and $R=\gl_{rg}$, then
$ST^*R=\gl_{st\an rg}$ is also in $\frak N(J)$ because $st\an rg\in E$.
Since $\frak N(J)$ is generated by $\{\gl_x:x\in E\}$, it follows 
that $ST^*R\in\frak N(J)$ for any 
three elements $S,T,R$ of $\frak N(J)$.  \qed

\medskip
{\remark Special cases of the above result are proved by Chu and Lau
in \cite{chulau} (see Propositions 3.2.10 and 3.3.9.)}

\bigskip 

We now pass from $\vn(G)$ to $\cl B (L^2(G))$:
The algebra $M^{\cb}A(G)$ acts on $\cl B(L^2(G))$ {\it via} the maps 
$S_{N(\gs)},\, \gs\in M^{\cb}A(G)$ (see \cite{bf} and \cite{j}), 
and this action is an extension of the action of $M^{\cb}A(G)$ on $\vn (G)$: 
when $T\in\vn (G)$ and $\sigma\in M^{\cb}A(G)$, 
we have $S_{N(\gs)}(T)=\gs\cdot T$.
Hence, letting
\[
\tilde{\frak N}(\Gs) =
 \{T\in\cl B(L^2(G)): S_{N(\gs)}(T)=0  , \ \mbox{ for all } \gs\in \Gs\},
\]
we have $\frak{N}(\Sigma)=\tilde{\frak N}(\Gs)\cap \vn(G).$ 

The following is analogous to Proposition \ref{p_njan}; note, however,
that the dualities are different.

\begin{proposition}\label{new}  If $\Gs\subseteq M^{\cb}A(G)$,  
\[  \tilde{\frak{N}}(\Gs)=\tilde{\frak{N}}(\Gs A).\]
\end{proposition}
\proof The inclusion $\tilde{\frak{N}}(\Gs)\subseteq \tilde{\frak{N}}(\Gs A)$ 
follows as in the proof of Proposition \ref{p_njan}.  
To prove that 
$\tilde{\frak N}(\Gs A)\subseteq \tilde{\frak{N}}(\Gs)$,
let $T\in \tilde{\frak{N}}(\Gs A)$; then $S_{N(v\gs)}(T)=0$ for all 
$\gs \in\Gs$ and $v \in A(G)$. Thus, if $h\in T(G)$, 
\[\du{S_{N(\gs)}(T)}{N(v) h} =\du{T}{N(\gs v) h} = \du{S_{N(v\gs )}(T)}{ h} =  0\, .\]
Since the linear span of the set $\{N(v) h: v \in A(G), h \in T(G) \}$ is dense in  $T(G)$,  
it follows that $S_{N(\gs)}(T)=0$ and so $T \in \tilde{\frak{N}}(\Gs)$. \qed

\begin{proposition}\label{prop2}
For every closed ideal $J$ of $A(G),\quad \tilde{\frak N}(J)= \Bim(J^\bot) $. 
\end{proposition}	
\proof If $T\in\cl B(L^2(G)), h\in T(G)$  and $u\in A(G)$ then 
\[\langle S_{N(u)}(T),h\rangle = \langle T, N(u)h\rangle .\]
By \cite[Proposition 3.1]{akt}, 
$\Sat(J)$ is the closed linear span of $\{N(u)h: u\in J , h\in T(G)\}$.
We conclude that  $T\in (\Sat(J))^\bot$ if and only if $S_{N(u)}(T) = 0$ for all 
$u\in J$, {\it i.e.} if and only if $T\in \tilde{\frak{N}}(J)$. 
By Theorem \ref{th_satlcg}, $(\Sat(J))^\bot=\Bim(J^\bot)$, and the proof is complete.
$\qquad\Box$

\begin{theorem}\label{thbimn} 
For any subset $\Gs$ of $M^{\cb}A(G)$, 
\[  \tilde{\frak N}(\Gs)=  \Bim(\frak{N}(\Gs)).\]
\end{theorem}
\proof 
It  follows from relation (\ref{eq_prean}) that 
$\Bim((\Gs  A)^\bot) = \Bim(\frak{N}(\Gs))$.
But $\Bim((\Gs  A)^\bot)=\tilde{\frak N}(\Gs A)$ from Proposition \ref{prop2} and 
$\tilde{\frak N}(\Gs A)=\tilde{\frak N}(\Gs)$ from Proposition \ref{new}.
 \qed

\medskip \smallskip

More can be said when the zero set $Z(\Gs)$ is a subgroup (or a coset) of $G$. 

\begin{lemma}\label{trito}
Let $J\subseteq A(G)$ be a closed ideal. Suppose that its zero set $E=Z(J)$ is 
a set of synthesis. Then
\begin{equation}\label{eq}
\tilde{\frak N}(J) 
=\overline{\sspp\{M_g\gl_x:x\in E,g\in L^\infty(G)\}}^{w*}
\end{equation}
\end{lemma}
\proof
By Theorem \ref{thbimn}, $\tilde{\frak N}(J) = \Bim(\frak N(J))$ 
and thus, by Lemma \ref{proto},
$\tilde{\frak N}(J)$ is the weak* closed linear span of
the monomials of the form 
$M_f\gl_sM_g$ where  $f,g\in L^\infty(G)$ and $s\in E$. But, because of the 
commutation relation
$\gl_sM_g=M_{g_s}\gl_s \ \ (\mbox{where } g_s(t)=g(s\an t))$,
we may write $M_f\gl_sM_g=M_\phi\gl_s$ where $\phi=fg_s\in L^\infty(G)$.\qed

\begin{theorem}\label{tetarto}
Let $J\subseteq A(G)$ be a closed ideal. Suppose that its zero set $E=Z(J)$ is 
the  coset of a closed subgroup of $G$. Then
$\tilde{\frak N}(J)$ is a (weak* closed) TRO. In particular if $E$ is a closed subgroup
then  $\tilde{\frak N}(J)$ is a von Neumann subalgebra of $\cl B(L^2(G))$ and 
\[
\tilde{\frak N}(J)=(\cl D\cup\frak N(J))''=(\cl D\cup\{\gl_x:x\in E\})''. 
\]
\end{theorem}

\proof  
As in the proof of Proposition \ref{deutero}, we may take $E=Hg$. 
By Lemma \ref{trito}, it suffices to check the TRO relation for monomials 
of the form $M_f\gl_{sg}$; but, by the commutation relation, 
triple products 
$(M_f\gl_{sg})(M_g\gl_{tg})^*(M_h\gl_{rg})$ 
of such monomials  may be written in the form
 $M_\phi\gl_{st\an rg}$ and so belong to  $\tilde{\frak N}(J)$ when $sg,tg$ and $rg$
are in the coset $E$. 
Finally, when $E$ is a closed subgroup, the last equalities follow 
from relation (\ref{eq}) and the bicommutant theorem.
\qed

\medskip

We next extend the notions of $\sigma$-harmonic functionals \cite{chulau} 
and operators \cite{neurun} to jointly harmonic functionals and operators:

\begin{definition}\label{d_jh}
Let $\Gs\subseteq M^{\cb}A(G)$. 
An element  $T\in \cl \vn(G)$ will be 
called a \emph{$\Gs$-harmonic functional } if 
 $\gs\cdot T=T$ for all $\gs\in\Gs$. We write   $\cl H_{\Sigma}$
for the set of all {$\Gs$-harmonic} functionals.

An operator $T\in \cl B(L^2(G))$ will be called \emph{$\Gs$-harmonic} if 
$S_{N(\gs)}(T)=T$ for all $\gs\in\Gs$. We write   $\widetilde{\cl H}_{\Sigma}$
for the set of all {$\Gs$-harmonic} operators. 
\end{definition}

Explicitly, if  $\Gs'=\{\sigma -\mathbf 1:\sigma\in\Gs\}$, 
\begin{align*}
\cl H_{\Sigma} &= 
\{T\in \vn(G) : \gs\cdot T=T \;\text{for all }\; \gs\in\Gs\} = \frak N(\Gs') \\
\text{and }\quad
\widetilde{\cl H}_{\Sigma} \ &= 
\{T\in \cl B(L^2(G)) : S_{N(\gs)}(T)=T \;\text{for all }\; \gs\in\Gs\} = \tilde{\frak N}(\Gs').
\end{align*}
The following  is an immediate consequence of Theorem \ref{thbimn}.

\begin{corollary}\label{c_jho} 
Let $\Gs\subseteq M^{\cb}A(G)$. 
Then the weak* closed $\cl D$-bimodule $\Bim(\cl H_{\Gs})$ 
generated by $\cl H_{\Gs}$ coincides with $\widetilde{\cl H}_{\Gs}$. 
\end{corollary}

Let $\gs$ be a positive definite normalised function and $\Gs = \{\gs\}$.
In \cite[Theorem 4.8]{neurun}, the authors prove,
under some restrictions on $G$ or $\gs$ (removed in \cite{kalantar}), 
that $\widetilde{\cl H}_{\Gs}$ coincides with the von
Neumann algebra $(\cl D\cup\cl H_{\Gs})''$. 
We give a short proof of a more general result.

Denote by $P^1(G)$ the set of all positive definite normalised functions on $G$. 
Note that $P^1(G)\subseteq M^{\cb}A(G)$. 

\medskip 

\begin{theorem} 
Let $\Gs\subseteq P^1(G)$. 
The space   $\widetilde{\cl H}_{\Gs}$ is a von Neumann subalgebra of $\cl B(L^2(G))$,
 and   $\widetilde{\cl H}_{\Gs}=(\cl D \cup\cl  H_{\Gs})''$. 
\end{theorem}
\proof Note that  $\cl H_{\Gs}=\frak N(\Gs')=\frak N(\Gs' A)$ and 
$\widetilde{\cl H}_{\Gs} = \tilde{\frak N}(\Gs')= \tilde{\frak N}(\Gs' A)$.
Since $Z(\Gs')$ is a closed subgroup \cite[Proposition 32.6]{hr2}, 
it is a set of spectral synthesis \cite{tatsuuma2}. Thus the result 
follows from  Theorem \ref{tetarto}. 
\qed

{\remark It is worth pointing out that $\widetilde{\cl H}_{\Gs}$ 
has an abelian commutant,
since it contains a masa. In particular,  it is a type I, 
and hence an injective, von Neumann algebra.} 

\bigskip

In \cite[Theorem 4.3]{akt} it was shown that a weak* closed subspace
$\cl U\subseteq \cl B(L^2(G))$ is jointly invariant if and only if it is of the form 
$\cl U = \Bim(J^{\perp})$ for a  closed ideal $J\subseteq A(G)$.  By 
Proposition \ref{prop2}, $\Bim(J^{\perp})=\tilde{\frak{N}}(J)$,
giving another equivalent description. In fact,  the ideal $J$  may be replaced
by a subset of  $M^{\cb}A(G)$:

\begin{proposition}\label{th_eqc}  
Let $\cl U\subseteq \cl B(L^2(G))$ be a weak* closed subspace. 
The following are equivalent:     

(i) \ \ $\cl U$ is jointly invariant; 

(ii) \ there exists a closed ideal $J\subseteq A(G)$ such that $\cl U = \tilde{\frak{N}}(J)$;

(iii) \ there exists a subset $\Sigma\subseteq M^{\cb}A(G)$ such that 
$\cl U = \tilde{\frak{N}}(\Gs)$.

\end{proposition}
\begin{proof}
 We observed the implication (i)$\Rightarrow$(ii) above, and  
(ii)$\Rightarrow$(iii) is trivial. 
 Finally, (iii)$\Rightarrow$(i) follows from 
Theorem \ref{thbimn} and \cite[Theorem 4.3]{akt}.
\end{proof}

{\remark 
It might also be observed that every  weak* closed  
jointly invariant subspace $\cl U$ is of the form 
$\cl U = \widetilde{\cl H}_{\Gs}$
for some $\Gs\subseteq M^{\cb}A(G)$.}

\bigskip

We end this section with a discussion on the ideals of the form $\Sigma A$:
If $J$ is a closed ideal of $A(G)$, then $J A\subseteq J$; thus, by (\ref{eq_prean}) 
and Proposition \ref{p_njan},
$J^\bot\subseteq \frak N(J)$ and therefore 
$\Bim (J^\bot)\subseteq\tilde{\frak N}(J)$, 
since $\tilde{\frak N}(J)$ is a $\cl D$-bimodule and contains ${\frak N}(J)$. 
The equality 
$J^\bot= \frak N(J)$ holds if and only if $J$ is 
generated by its compactly supported elements, equivalently if 
$J=JA$ (see Remark \ref{remideal}). Indeed, by 
Proposition \ref{p_njan} we have $\frak N(J)=\frak N(JA)= (JA)^\bot$ and so  
the equality $J^\bot= \frak N(J)$ is equivalent to $J^\bot= (JA)^\bot$.
Interestingly, the inclusion $\Bim (J^\bot)\subseteq\tilde{\frak N}(J)$ 
is in fact always an equality (Proposition \ref{prop2}). 

We do not know whether all closed ideals of $A(G)$ are of the form $\Sigma A$.
They certainly are when $A(G)$ satisfies {\em Ditkin's condition at infinity}
\cite[Remark 5.1.8 (2)]{kaniuth}, 
namely if 
every $u\in A(G)$ is the limit of a sequence $(uv_n)$, with $v_n\in A_c(G)$. 
Since $A_c(G)$ is dense in $A(G)$, this is equivalent to the condition that 
every $u\in A(G)$ belongs to the closed ideal $\overline{uA(G)}$.

This condition has been used before (see for example \cite{kl}). 
It certainly holds whenever $A(G)$ has a weak form of approximate identity;
for instance, when $G$ has the approximation property (AP) of 
Haagerup and Kraus \cite{hk} and a fortiori when $G$ is amenable. 
It also holds for all discrete groups.
See also the discussion in Remark 4.2 of \cite{lt} and the one
following Corollary 4.7 of \cite{akt}.

\section{Annihilators and Supports}\label{s}

In this section, given a set $\cl A$ of operators on $L^2(G)$,  we study the ideal of
all  $u\in A(G)$ which act trivially on $\cl A$; its zero set is the $G$-support 
of $\cl A$; we relate this to the $\go$-support of  $\cl A$ defined in \cite{eks}.
 
In \cite{eymard}, Eymard introduced, for $T\in\vn (G)$, the ideal $I_T$ of all
$u\in A(G)$ satisfying $u\cdot T=0$. We generalise this by defining, 
for a subset  $\cl A$ of $\cl B(L^2(G))$,  
\[
I_\cl{A} =\{u \in  A(G): S_{N(u)}(\cl A)=\{0\}. \}
\]
It is easy to verify that  $I_\cl{A}$ is a closed ideal of  $A(G)$.

Let $\cl U(\cl A)$ be the smallest weak* closed jointly invariant 
subspace containing $\cl A$. 
We next prove that $\cl U(\cl A)$ coincides with the set $\tilde{\frak N}(I_\cl{A})$ 
of all $T\in\cl B(L^2(G))$ satisfying $S_{N(u)}(T)=0$ for all $u \in I_\cl{A}$.  

\begin{proposition} \label{13} 
Let $\cl A\subseteq\cl B(L^2(G))$.
If $\gs\in  M^{\cb}A(G)$ then $S_{N(\gs)}(\cl A)=\{0\}$ if and only if  
$S_{N(\gs)}(\cl U(\cl A))=\{0\}$.
Thus, $I_\cl{A}=I_\cl{U(A)}$.
\end{proposition} 
\proof 
Recall that 
$$\cl U(\cl A) = 
\overline{\sspp\{S_w(\rho_r T \rho_r^*) : T\in \cl A, w\in \frak{S}(G), r\in G\}}^{w^*}.$$
The statement now follows immediately from the facts that 
$S_{N(\gs)}\circ S_w= S_w\circ S_{N(\gs)}$ for all $w\in \frak S(G)$
and $S_{N(\gs)}\circ {\rm Ad}_{\rho_r}= {\rm Ad}_{\rho_r}\circ S_{N(\gs)}$ 
for all $r\in G$. 
The first commutation relation is obvious, 
and the second one can be seen as follows:
Denoting by  $\theta_r$ the predual of the
map ${\rm Ad}_{\rho_r}$, for all $h\in T(G)$ we have  
$\theta_r(N(\gs)h) = N(\gs)\theta_r(h)$ since  $N(\gs)$ is right invariant and so
\begin{align*}
\du{S_{N(\gs)}(\rho_rT\rho_r^*)}{h} &= \du{\rho_rT\rho_r^*}{N(\gs)h}
 = \du{T}{\theta_r(N(\gs)h)} \\
& = \du{T}{N(\gs)\theta_r(h)} = \du{S_{N(\gs)}(T)}{\theta_r(h)} \\
&= \du{\rho_r(S_{N(\gs)}(T))\rho_r^*}{h}.
\end{align*}
Thus $S_{N(\gs)}(\rho_rT\rho_r^*)=\rho_r(S_{N(\gs)}(T))\rho_r^*$. 
\qed

\begin{theorem} \label{prop16} 
Let $\cl A\subseteq\cl B(L^2(G))$.  The bimodule 
$\tilde{\frak N}(I_\cl{A})$ coincides with the smallest weak* closed jointly 
invariant subspace $\cl U(\cl A)$ of $\cl B(L^2(G))$ containing $\cl A$. 
\end{theorem}
\proof 
Since  $\cl{U(A)}$ is weak* closed and jointly invariant, by \cite[Theorem 4.3]{akt} it 
equals $\Bim(J^\bot)$, 
where $J$ is the closed ideal of $A(G)$ given by
\[J=\{u\in A(G): N(u)\chi_{L\times L} \in (\cl{U(A)})_\bot
\;\text{for all  compact $L\subseteq G$}\}.\]

We show that $J\subseteq I_\cl{A}$. Suppose $u\in J$;  then, 
for all $w\in\frak S(G)$ and all $T\in\cl A$, since $S_w(T)$ is in $\cl{U(A)}$, 
by Theorem \ref{th_satlcg} 
it annihilates $ N(u)\chi_{L\times L}$ for every  compact $L\subseteq G$.
It follows that 
$$\du{S_{N(u)}(T)}{w\chi_{L\times L}} = \du{T}{N(u)w\chi_{L\times L}} =\du{S_w(T)}{N(u)\chi_{L\times L}} = 0$$
for all  $w\in\frak S(G)$ and all compact $L\subseteq G$. 
Taking $w=f\otimes\bar g$
with $f,g\in L^\infty(G)$  supported in $L$, this yields
\[
\sca{S_{N(u)}(T)f,g} =
\du{S_{N(u)}(T)}{w\chi_{L\times L}} = 0
\]
for all  compactly supported $f,g\in L^\infty(G)$ and therefore $S_{N(u)}(T)=0$.
Since this holds for all $T\in\cl A$, we have shown that $u\in I_\cl{A}$.

\smallskip It follows that $\cl{U(A)}=\Bim(J^\perp)\supseteq \Bim(I_\cl{A}^\perp)$.
But $\Bim(I_\cl{A}^\perp)=\tilde{\frak N}(I_\cl{A})$ by Proposition \ref{prop2}, 
and this space is clearly jointly invariant and weak* closed. 
Since it contains $\cl A$, it also contains $\cl{U(A)}$ and so
\[\cl{U(A)}=\Bim(J^\perp)= \Bim(I_\cl{A}^\perp)=\tilde{\frak N}(I_\cl{A}). \qquad\Box\]

\bigskip

\noindent\textbf{Supports of functionals and operators}                
In \cite{neurun}, the authors generalise the notion of support of an element of 
 $\vn(G)$ introduced by Eymard \cite{eymard} by defining, for an arbitrary  
$T\in\cl B(L^2(G))$, 
\[ \suppG T := 
\{x\in G : u(x) = 0 \;\text{for all  $u\in A(G)$ with }\; S_{N(u)}(T) = 0\}.\]
Notice that $\suppG T$ coincides with the zero set of the ideal $I_T$ (see also 
\cite[Proposition 3.3]{neurun}). More generally, let us define the {\em $G$-support}
of a subset $\cl A$ of $\cl B(L^2(G))$ by 
\[\suppG(\cl A) = Z(I_\cl{A}).\]
When $\cl A\subseteq \vn(G)$, then $\suppG(\cl A)$ is just the support 
of $\cl A$ considered as a set of functionals on $A(G)$ as in \cite{eymard}.

The following is proved in \cite{neurun} under the assumption 
that $G$ has the approximation property of Haagerup and Kraus \cite{hk}:

\begin{proposition}
Let $T\in\cl B(L^2(G))$. Then $\suppG(T)=\emptyset$ if and only if $T=0$.
\end{proposition}
\proof It is clear that the empty set is the $G$-support of the zero operator.
 Conversely, 
suppose $\suppG(T)=\emptyset$, that is, $Z(I_T)=\emptyset$. This implies
that $I_T=A(G)$ (see \cite[Corollary 3.38]{eymard}). 
Hence $S_{N(u)}(T)=0$ for all $u\in A(G)$, and so for all $h\in T(G)$ we have
\[
\du{T}{N(u)h}= \du{S_{N(u)}(T)}{h}=0.
\]
Since  the linear span of $\{N(u)h:u\in A(G), h\in T(G)\}$ is dense in $T(G)$, 
it follows that $T=0.$
\qed 

\begin{proposition}\label{propsame}
The $G$-support of a subset $\cl A\subseteq\cl B(L^2(G))$ is the same as 
 the $G$-support of the smallest weak* closed jointly invariant subspace $\cl{U(A})$ 
containing $\cl A$.
\end{proposition} 
\proof Since $I_\cl{A}=I_{\cl{U(A})}$ (Proposition \ref{13}), this is immediate. \qed

\medskip

The following proposition shows that the $G$-support of a subset 
$\cl A\subseteq\cl B(L^2(G))$ is in fact the support of a space of linear
functionals on $A(G)$ (as used by Eymard): it  
can be obtained either by first forming 
the ideal $I_\cl{A}$ of all $u\in A(G)$  `annihilating' $\cl A$ 
(in the sense  that $S_{N(u)}(\cl A)=\{0\}$)
and then taking the support of the annihilator of $I_\cl{A}$ in $\vn(G)$; alternatively,
it can be obtained by forming the smallest weak* closed jointly invariant subspace
 $\cl{U(A})$
containing $\cl A$ and then considering the support of the set of 
all  the functionals on $A(G)$ which are contained in $\cl{U(A})$.

\begin{proposition}\label{propsame2}
The $G$-support of a subset $\cl A\subseteq\cl B(L^2(G))$ coincides with the supports
of the following  spaces of functionals on $A(G)$:

(i) \ the space $I_\cl{A}^\bot\subseteq\vn(G)$  

(ii) the space $\cl{U(A})\cap\vn(G)=\frak N(I_\cl{A})$. 
\end{proposition}
\proof 
By Proposition \ref{prop2} and Theorem \ref{prop16},
\[\cl{U(A})= \tilde{\frak N}(I_\cl{A})=\Bim( I_\cl{A}^\bot).\]
Since the $\cl D$-bimodule $\Bim( I_\cl{A}^\bot)$ is jointly invariant, 
it coincides with $\cl U(I_\cl{A}^\bot)$. 
Thus $\cl{U(A})=\cl U(I_\cl{A}^\bot)$ and so Proposition \ref{propsame} gives  
$\suppG(\cl A)=\suppG( I_\cl{A}^\bot)$, proving part (i).

Note that                     
$\cl U(\frak N(I_\cl{A}))=\Bim(\frak N(I_\cl{A}))=\tilde{\frak N}(I_\cl{A})$ and so
$\cl U(\frak N(I_\cl{A}))=\cl{U(A})$. Thus by Proposition \ref{propsame}, 
$\frak N(I_\cl{A})$ and $\cl A$ 
have the same support.  Since 
$\cl{U(A})\cap\vn(G)=\tilde{\frak N}(I_\cl{A})\cap\vn(G)=\frak N(I_\cl{A})$,
part (ii) follows. \qed

\bigskip

We are now in a position to relate the $G$-support of a set of operators to their
$\omega$-support as introduced in \cite{eks}.

\begin{theorem}\label{312} 
Let $\cl U\subseteq\cl B(L^2(G))$ be a weak* closed jointly invariant subspace. Then 
\begin{align*}
\suppo(\cl U) &\cong (\suppG(\cl U))^*.
\end{align*}
In particular, the $\omega$-support of a jointly invariant subspace is marginally
equivalent to a topologically closed set. 
\end{theorem}
\proof 
Let $J = I_\cl{U}$. By definition, $\suppG(\cl U) = Z(J)$. By 
the proof of Theorem \ref{prop16},
$\cl U = \Bim(J^\bot)$, and hence, by Theorem \ref{th_satlcg}, $\cl U = (\Sat J)^\bot$.
By \cite[Section 5]{akt}, 
$\suppo(\cl U)=\nul(\Sat J)=(Z(J))^*$, where 
$\nul (\Sat J)$ is the largest, up to marginal equivalence, 
$\omega$-closed subset $F$ of $G\times G$
such that $h|_F = 0$ for all $h\in\Sat J$ (see \cite{st1}).  
The proof is complete.
\qed

\begin{corollary}\label{c_nss}
Let $\Gs\subseteq M^{\cb}A(G)$. Then  
\[
\suppo \tilde{\frak{N}}(\Gs) \cong Z(\Gs)^*.
\]
If  $Z(\Sigma)$ satisfies 
spectral synthesis, then $\tilde{\frak{N}}(\Gs) = \frak{M}_{\max}(Z(\Sigma)^*)$. 
\end{corollary}
\begin{proof} From Theorem \ref{thbimn}, we know that
$\tilde{\frak{N}}(\Gs)=\Bim((\Gs A)^\bot)=\tilde{\frak{N}}(\Gs A)$
and so $\suppo \tilde{\frak{N}}(\Gs) \cong Z(\Gs A)^*$
by \cite[Section 5]{akt}.
But $Z(\Gs A)=Z(\Gs)$ as can easily be verified (if $\gs(t)\ne 0$ 
there exists $u\in A(G)$ so that $(\gs u)(t)\ne 0$; the converse is trivial).

The last claim follows from 
the fact that, when $Z(\Gs)$ satisfies 
spectral synthesis, there is a unique weak* closed $\cl D$ bimodule whose 
$\go$-support is  $Z(\Gs)^*$ (see \cite[Theorem 4.11]{lt} or the proof of
\cite[Theorem 5.5]{akt}).  
 \end{proof} 

\smallskip

Note that when $\Gs\subseteq P^1(G)$, the set $Z(\Gs)$ satisfies 
spectral synthesis.
\medskip

The following corollary is a direct consequence of Corollary \ref{c_nss}.

\begin{corollary}\label{corsyn}
Let $\Gs\subseteq M^{\cb}A(G)$ and $\Gs'=\{\mathbf 1-\gs:\gs\in\Gs\}$. 
 If $Z(\Gs') $ 
is a set of spectral synthesis, then 
$\widetilde{\cl H}_{\Gs} = \frak{M}_{\max}(Z(\Gs')^*)$.
\end{corollary}	

\begin{corollary} Let  $\Omega$ be a subset of $G\times G$ which 
is invariant under all maps $(s,t)\to (sr,tr), \, r\in G$.
 Then $\Omega$ is marginally equivalent to an $\omega$-closed set if and only if 
it is marginally equivalent to a topologically closed set. 
\end{corollary}
\proof A topologically closed set is of course $\omega$-closed. For the converse, let
$\cl U=\frak{M}_{\max}(\Omega)$, so that
$\Omega\cong\suppo(\cl U)$. Note that $\cl U$ is a 
weak* closed jointly invariant space. 
Indeed, since $\Omega$ is invariant, for every $T\in\cl U$ 
the operator $T_r=: \rho_rT\rho_r^*$ is supported in $\Omega$ 
and hence is in $\cl U$. Of course $\cl U$ is invariant under all Schur multipliers. 
By  Theorem \ref{312},  
$\suppo(\cl U)$ is marginally equivalent to a closed set. \qed

\begin{theorem}\label{th_compsa}
Let $\cl A\subseteq \cl B(L^2(G))$. Then $\suppG(\cl A)$ is 
the smallest closed subset $E\subseteq G$ such that $E^*$ marginally contains 
$\suppo(\cl A)$.
\end{theorem}

\proof 
{Let $\cl U=\cl{U(A})$ be the smallest 
jointly invariant weak* closed subspace containing $\cl A$. 
Let $Z=Z(I_\cl{A})$; by definition, $Z=\suppG\cl A$. But 
$\suppG\cl A=\suppG\cl U=Z$ (Proposition \ref{propsame}) 
and so $\suppo\cl{U}\cong Z^*$ by Theorem \ref{312}.

Thus $Z^*$ does marginally contain $\suppo(\cl A)$.

On the other hand, let 
$E\subseteq G$ be a closed set such that $E^*$ marginally contains 
$\suppo(\cl A)$. Thus any operator $T\in\cl A$ is supported in
 $E^*$. But since   $E^*$  is invariant, $\rho_rT\rho_r^*$ is also supported in $E^*,$
for every $r\in G$. Thus $\cl U$ is supported in $E^*$.

This means  that $Z^*$ is marginally contained in $E^*$;
that is, there is a null set $N\subseteq G$ 
such that $Z^*\setminus E^*\subseteq (N\times G)\cup (G\times N)$. 
We claim that $Z\subseteq E$. To see this, assume, by way of contradiction, 
that there exists $s\in Z\setminus E$.
Then the `diagonal' $\{(r,sr):r\in G\}$ is a subset of 
$Z^*\setminus E^*\subseteq (N\times G)\cup (G\times N)$. 
It follows that for every $r\in G$, either $r\in N$ or $sr\in N$, which means that 
$r\in s\an N$. Hence $G\subseteq N\cup s\an N$, which is a null set. 
This contradiction shows that $Z\subseteq E$. 

\qed

\medskip

We note that for subsets $\cl S$ of $\vn(G)$ the  relation 
$\suppo (\cl{S})\subseteq (\suppG(\cl S))^*$ 
is  in \cite[Lemma 4.1]{lt}. 

\medskip

In \cite{neurun} the authors define, for a closed subset $Z$ of $G$,
the set 
\[
\cl B_Z(L^2(G)) = \{T\in\cl B(L^2(G): \suppG(T)\subseteq Z\}. 
\] 

\begin{corollary}\label{rem38}
If  $Z\subseteq G$ is closed, the  set $\cl B_Z(L^2(G))$ consists of all $T\in\cl B(L^2(G))$
 which are $\go$-supported in $Z^*$; that is,  $\cl B_Z(L^2(G))=\frak{M}_{\max}(Z^*)$. 
In particular, this space is a reflexive jointly invariant subspace. 
\end{corollary}

\proof If $T$ is $\go$-supported in $Z^*$, then by Theorem \ref{th_compsa}, 
$\suppG(T)\subseteq Z$. 

Conversely if $\suppG(T)\subseteq Z$ then
$\suppG(\cl U(T))\subseteq Z$ by Proposition \ref{propsame}. 
But, by Theorem \ref{312},
$\suppo(\cl U(T)) \cong (\suppG(\cl U(T)))^*\subseteq Z^*$ 
 and so $T$ is $\go$-supported in $Z^*$. \qed

\begin{remark}  \label{last}
The $\omega$-support 
$\suppo(\cl A)$ of a set $\cl A$ of operators is more \lq sensitive' 
than $\suppG(\cl A)$ in that 
it encodes more information about $\cl A$. Indeed, $\suppG(\cl A)$ 
only depends  on the 
(weak* closed) jointly invariant subspace generated by $\cl A$, while 
$\suppo(\cl A)$ depends on the (weak* closed) masa-bimodule
 generated by $\cl A$.
\end{remark}

{\example 
Let $G=\bb Z$ and $\cl A=\frak{M}_{\max}\{ (i,j):i+j \in\{0,1\}\}$. 
The $\go$-support of $\cl A$ is of course the two-line set $\{ (i,j):i+j \in\{0,1\}\}$,
while its $G$-support is $\bb Z$  
which gives no information about $\cl A$.}

Indeed, if $E\subseteq\bb Z$ contains $\suppG(\cl A)$, then by Theorem 
\ref{th_compsa} $E^*=\{(n,m)\in\bb Z\times\bb Z:m-n\in E\}$ must contain 
$\{ (i,j):i+j \in\{0,1\}\}$. Thus for all $n\in\bb Z$, since $(-n,n)$ and  $(-n,n+1)$ are
in $\suppo(\cl A)$ we have $n-(-n)\in E$ and  $n+1-(-n)\in E$; hence 
$\bb Z\subseteq E$.

\def\cprime{$'$} \def\cprime{$'$}

\end{document}